\documentclass[10pt,reqno]{amsart}
\usepackage{verbatim,amscd, 
amsmath,amsthm,amssymb,amsfonts}
\usepackage{xcolor}

\numberwithin{equation}{section} 
\newtheorem{theorem}{Theorem}[section]
\newtheorem{corollary}[theorem]{Corollary}
\newtheorem{lemma}[theorem]{Lemma}
\newtheorem{proposition}[theorem]{Proposition}
\theoremstyle{definition}

\newtheorem{remark}{Remark}[section]
\newtheorem*{conjecture}{Conjecture}
\newtheorem*{acknow}{Acknowledgments}

\begin{document}
\title[Hamiltonian non-displaceability of Gauss images]{Hamiltonian non-displaceability  of Gauss images of isoparametric hypersurfaces}
\author{Hiroshi Iriyeh}
\address{Mathematics and Infomatics, College of Science, Ibaraki University,
Mito, 310-8512, Japan.
}
\email{hiroshi.irie.math@vc.ibaraki.ac.jp}
\author{Hui Ma}
\address{Department of Mathematical Sciences,
Tsinghua University, Beijing 100084, P.R. China}
\email{hma@math.tsinghua.edu.cn}
\author{Reiko Miyaoka}
\address{Mathematical Institute, Tohoku University, Sendai 980-8578, Japan}
\email{r-miyaok@m.tohoku.ac.jp}
\author{Yoshihiro Ohnita}
\address{Osaka City University Advanced Mathematical Institute \& Department of Mathematics, Osaka City University, 
Osaka, 558-8585, Japan}
\email{ohnita@sci.osaka-cu.ac.jp}
\thanks{
The authors were partly supported by JSPS Grant-in-Aid for Young Scientists (B) No.~24740049, 
NSFC grant No.~11271213, 
JSPS Grant-in-Aid for Scientific Research (B) No.~15H03616, 
JSPS Grant-in-Aid for Scientific Research (S) No.~23224002 and (C) No.~15K04851,
respectively.}


\maketitle

\begin{abstract}
In this article we study the Hamiltonian non-displaceability 
of Gauss images of isoparametric hypersurfaces in the spheres 
as Lagrangian submanifolds embedded in complex hyperquadrics.
\end{abstract}

\section{Introduction}
\label{Sec:Intro}

The purpose of this article is to build a bridge between submanifold theory 
and symplectic geometry. 
Here, by submanifolds we mean  the family of isoparametric hypersurfaces, 
originated from geometric optics in Italy, 
which is systematically developed by \'{E}lie Cartan. 
The latter is the Lagrangian intersection theory proposed by  
A.~Weinstein and V.I.~Arnold. 
Needless to say, A.~Floer's contribution is quite important \cite{Floer88}.  
Recently, there has been intensive progress for the Lagrangian Floer theory 
of toric fibers.
In order to proceed the study, we need, on  one hand, 
a lot of concrete examples, and on the other hand, the common properties 
such examples share. 

The isoparametric hypersurfaces are compact oriented hypersurfaces embedded 
in the standard sphere with constant principal curvatures. 
They provide a rich class containing infinitely many 
both homogeneous and non-homogeneous examples. 
Their topological and differential geometric data are well investigated, 
and manageable in various calculations. 
M\"{u}nzner shows that the number $g$ of distinct constant principal curvatures 
$\kappa_1>\dots>\kappa_g$ 
must be $1, 2, 3, 4$ or $6$, and the multiplicities $m_i$ of $\kappa_i$ 
are related by $m_i=m_{i+2}$ $(i\ \mathrm{mod}\, g)$ 
\cite{Muenzner1}. 
Without loss of generality, we may assume $m_1\le m_2$.  
There are infinitely many non-homogeneous 
examples with $g=4$, \cite{Ozeki-Takeuchi}, \cite{Ferus-Karcher-Muenzner}. 
The theory has already been applied to various mathematical problems 
(e.g., \cite{Wang}, \cite{Ge-Xie}, \cite{Tang-Yan}).

Through the Gauss map $\mathcal{G}$, 
isoparametric hypersurfaces $N^n$ in  $S^{n+1}$ yield a nice class of Lagrangian submanifolds $L^n=\mathcal{G}(N^n)$
embedded in the complex hyperquadric $Q_n(\mathbb{C})$ $(n\geq 2)$, 
which is a rank two Hermitian symmetric space of compact type. 
B.~Palmer \cite{Palmer97} pointed out the minimality of  
$L$, and the second and the fourth authors 
discussed the Hamiltonian stability and related properties of $L$ 
\cite{Ma-Ohnita1}--\cite{Ma-OhnitaII}. 

As  a first step to obtain the Floer homology of $L$, which has been achieved 
in the cases  $g=1,2$ \cite{OhIII}, \cite{IST}, we investigate the Hamiltonian non-displaceability of  $L$. 
Here, a compact Lagrangian submanifold $L$ embedded  in a symplectic manifold $(M, \omega)$ is said to be 
\emph{Hamiltonian non-displaceable}, if $L\cap \varphi (L)\neq \emptyset$ for any $\varphi$ 
belonging to the group $\mathrm{Ham}(M,\omega)$ of all Hamiltonian diffeomorphisms of $M$. 

Our main result is as follows:

\begin{theorem}\label{Main Thm}
Let $N^n$ be a compact oriented isoparametric hypersurface of the standard sphere 
$S^{n+1}(1)$ in $\mathbb{R}^{n+2}$.
Then its Gauss image $L^n=\mathcal{G}(N^n)$ is Hamiltonian non-displaceable in $(Q_n(\mathbb{C}),\omega_{\textup{std}})$ 
except for the following few remaining cases: 
\[
\begin{array}{ll}
(g,n,m_1,m_2)=(3,3,1,1), 
&N=
\frac{SO(3)}{\mathbb{Z}_2+\mathbb{Z}_2},
\\
(g,n,m_1,m_2)=(4, 2k+2, 1,k),
&N=
\frac{SO(2)\times SO(k+2)}{\mathbb{Z}_2\times SO(k)}\  (k\geq 1),
\\
(g,n,m_1,m_2)=(6,6,1,1), 
&N=
\frac{SO(4)}{\mathbb{Z}_2+\mathbb{Z}_2}, 
\end{array}
\]
where $\omega_{\textup{std}}$ denotes the standard induced K\"ahler form of 
$Q_n(\mathbb{C})\subset {\mathbb C}P^{n+1}$. 
\end{theorem}

Here we should mention the preceding results and a remark. 
We know that any Gauss image $L= \mathcal{G}(N)$ is monotone. 
A compact monotone Lagrangian submanifold $L$ in $(M,\omega)$ is said to be \emph{wide} if the Floer homology $HF(L)$ satisfies 
\[
HF(L)\cong H_{*}({L;\mathbb{Z}_2})\otimes \Lambda,
\] 
and \emph{narrow} if $HF(L)=0$.
See Sections \ref{Sec:Gauss images}  and \ref{Sec: Damian} for details. 
Notice here if $L$ is wide, then $L$ is Hamiltonian non-displaceable.
Y.G.~Oh \cite{OhIII} showed that $L$ is wide for any real form $L$ of a compact irreducible Hermitian symmetric space 
$M$ (see also \cite{IST}). 
In particular, all real forms of $Q_n(\mathbb{C})$, consisting of the Lagrangian sphere $S^n$ and the real quadric 
$(S^k\times S^{n-k})/\mathbb{Z}_2$ ($1\le k\le n-1$), are wide. 
Note that they coincide with the Gauss image of isoparametric hypersurfaces 
with  $g=1$ and $2$, respectively. 
Moreover, 
when $g=3$ and $n=6,12$ or $24$, the Gauss image $L$ is wide, although $L$ is not a real form (Corollary \ref{cor:g=3}).
It is an interesting future problem to decide  $L$ is wide or not in the remaining cases.

The paper is organized as follows: 
In Section \ref{Sec:Gauss images}, we prepare fundamental properties of the Gauss images of isoparametric hypersurfaces 
in the standard sphere, especially the formula of their minimal Maslov number (Proposition \ref{MinMaslovGaussImage}). 
In Section \ref{Sec: Damian}, we recall Damian's lifted Floer homology and its spectral sequence \cite{Damian}, 
which is our main tool in the later sections.  
In Section \ref{Sec: g=3}, we examine the topology of Gauss images when $g=3$.
As a result we obtain that the Gauss image with $g=3$ is a $\mathbb{Z}_2$-homology sphere,
and it is wide if $n=6$, $12$ or $24$.
In Section \ref{Sec: g=4} and Section \ref{Sec: g=6}, 
we prove the Hamiltonian non-displaceability of Gauss images with $g=4$ or $6$ and $m_1\geq 2$, respectively.
In the final section, we propose several open problems and a conjecture.

Throughout this article any manifold is smooth and connected.

\begin{acknow}
This work was done mostly during the authors' short visits
to RIMS in Kyoto, Japan in 2014 and 2015, and TSIMF in Sanya, China, in 2014.  
The authors would like to express their sincere gratitude to both institutions 
for the excellent research circumstance and generous support. 
\end{acknow}

\section{Gauss images of isoparametric hypersurfaces}
\label{Sec:Gauss images}

Let $N^n$ be an oriented hypersurface embedded in the unit standard
sphere $S^{n+1}(1) \subset {\mathbb{R}}^{n+2}$. 
Denote by ${\mathbf x}$ the position vector of points of $N$ 
and ${\mathbf n}$ the unit normal vector field of $N$ in $S^{n+1}(1)$. 
It is a fundamental fact in differential geometry
that the \emph{Gauss map} defined by
\[
{\mathcal G}:N^{n}\ni{p} \longmapsto {\mathbf x}(p)\wedge
{\mathbf n}(p)\cong [{\mathbf x}(p)+\sqrt{-1}{\mathbf n}(p)]\in
\widetilde{Gr}_{2}({\mathbb{R}}^{n+2}) \cong Q_{n}({\mathbb{C}})
\]
is always a Lagrangian immersion into the complex hyperquadric
$Q_n({\mathbb{C}})$. Here the complex hyperquadric $Q_n({\mathbb{C}})$
is identified with the real Grassmann manifold
$\widetilde{Gr}_{2}({\mathbb{R}}^{n+2})$ of oriented $2$-dimensional
vector subspaces of ${\mathbb{R}}^{n+2}$, which has a symmetric space
expression $SO(n+2)/(SO(2)\times SO(n))$.

It follows from \cite{Palmer97} that the Gauss map ${\mathcal G}:
N^n \rightarrow Q_n({\mathbb{C}})$ from an isoparametric hypersurface $N^n$  in $S^{n+1}(1)$
is a {\em minimal Lagrangian immersion} into $Q_n(\mathbb{C})$.
Moreover, the \lq\lq{Gauss image}\rq\rq of $\mathcal G$ is a compact
minimal Lagrangian submanifold 
$L^n={\mathcal G}(N^n) \cong N^n /{\mathbb Z}_g$ \emph{embedded} in $Q_n({\mathbb{C}})$, 
where 
${\mathcal G}:N^n\rightarrow {\mathcal G}(N^n)=L^n$ is the covering
map with the Deck transformation group ${\mathbb Z}_g$
\cite{Ma-Ohnita1}, \cite{Ma-OhnitaCONM2010}, \cite{Ohnita10}. 

Let us briefly recall the notion of monotone Lagrangian submanifold.
Given a Lagrangian manifold $L$ in a symplectic manifold $(M,\omega)$,  two homomorphisms
$I_{\mu,L}: \pi_2 (M,L) \rightarrow \mathbb{Z}$ and $I_\omega : \pi_2(M,L)\rightarrow \mathbb{R}$ are defined as follows.
For a smooth map $u: (D^2, \partial D^2) \rightarrow (M,L)$ in the class $A\in \pi_2(M, L)$,  
there is a unique trivialization, up to homotopy,
of the pull-back bundle $u^{*}TM \cong D^2 \times \mathbb{C}^{n}$ as a symplectic bundle.
This gives a map $\tilde{u}$ from $S^1=\partial D^2$ to $\Lambda(\mathbb{C}^n)$, 
the set of Lagrangian vector subspaces in $\mathbb{C}^n$.
By using the Maslov class $\mu\in H^1(\Lambda(\mathbb{C}^n), \mathbb{Z})$, one can define
$I_{\mu,L}(A):=\mu(\tilde{u})$. 
Next $I_\omega$ is defined by $I_{\omega}(A):=\int_{D^2}u^*\omega$.
Then a Lagrangian submanifold $L$ in $M$ is said to be {\em monotone}, 
if there exists a constant $\lambda >0$ such that $$I_{\mu, L}=\lambda I_{\omega}.$$ 
Denote by $N_L\in \mathbb{Z}_+$ the positive generator of the image of $I_{\mu, L}$.
We call $N_L$ {\em the minimal Maslov number} of $L$.

We obtain the following properties of the Gauss image:

\begin{proposition}[\cite{Ohnita10}, \cite{Ma-OhnitaJDG}]\label{MinMaslovGaussImage}
The Gauss image $L^n=\mathcal{G}(N^n)$ 
of an isoparametric hypersurface $N^n$  in $S^{n+1}$ is  a compact monotone 
Lagrangian submanifold embedded in $Q_n(\mathbb{C})$, and its
minimal Maslov number $N_L$ is given by
\begin{equation}\label{minimalMaslov}
N_L=\frac{2n}{g}= 
\left\{
\begin{array}{ll}
m_1+m_2, & \hbox{ if } g  \hbox{ is even},  \\
2m_1, & \hbox{ if } g \hbox{ is odd}. 
\end{array}
\right. 
\end{equation}
The Gauss image $\mathcal{G}(N^n)$  is orientable if and only if $2n/g$ is even.
\end{proposition}

\begin{remark} It is well-known that $1\leq N_L\leq 2n$ for a closed monotone Lagrangian submanifold $L$ 
in $Q_n(\mathbb{C})$.
The upper bound is sharp since the Lagrangian sphere $S^n\subset Q_n(\mathbb{C})$ 
attains $N_L=2n$.
The Gauss images of isoparametric hypersurfaces with $g=1$ and $g=2$ 
are $S^n$ and $(S^k\times S^{n-k})/\mathbb{Z}_2 (1\leq k\leq n-1)$, 
respectively. 
They provide real forms of $Q_n(\mathbb{C})$, which are totally geodesic. 
Notice that the minimal Maslov number of $(S^k\times S^{n-k})/\mathbb{Z}_2 \, (1\leq k\leq n-1)$ is equal to $n$.
We observe from  (\ref{minimalMaslov}) that Gauss images 
of isoparametric hypersurfaces with $g\geq 3$ provide examples of Lagrangian submanifolds in $Q_n(\mathbb{C})$ with  
relatively large minimal Maslov numbers (for the case of $\mathbb{C}P^n$, see \cite{Iriyeh14}). 
On the other hand, $N_L=2$ is taken only 
when  $m_1=m_2=1$, 
and such isoparametric hypersurfaces  are homogeneous 
\cite{Cartan}, \cite{Takagi}, \cite{Dorfmeister}. 
When $g\ge3$, all Gauss images with $N_L=2$ are given respectively \cite{Ma-OhnitaJDG}  by   
\[
L=
\left\{
\begin{array}{ll}
\frac{SO(3)}{(\mathbb{Z}_2+\mathbb{Z}_2)\mathbb{Z}_3}, &(g, m_1, m_2)=(3, 1, 1), \\
\frac{SO(2)\times SO(3)}{(\mathbb{Z}_2\times 1)\mathbb{Z}_4}, &(g, m_1, m_2)=(4, 1, 1),  \\
\frac{SO(4)}{(\mathbb{Z}_2+\mathbb{Z}_2)\mathbb{Z}_6},  &(g, m_1, m_2)=(6, 1, 1). 
\end{array}
\right.
\]
\end{remark}

\section{Damian's lifted Floer homology}
\label{Sec: Damian}

In this section we recall necessary facts from 
the Floer theory for monotone Lagrangian submanifolds.

\subsection{Lagrangian Floer homology}

Let $(M,\omega)$ be a closed symplectic manifold and let $L\subset (M,\omega)$ be 
a closed connected monotone embedded Lagrangian submanifold with $N_L\geq 2$.
Consider a Hamiltonian isotopy $\{\varphi _t\}_{0\leq t\leq 1}$ defined by a time-dependent Hamiltonian 
$H:[0,1]\times M \rightarrow \mathbb{R}$.
The time-one map $\varphi_1$ is called a {\it Hamiltonian diffeomorphism} of $M$.
Assume that $L \pitchfork \varphi _1(L)\neq \emptyset$. 
The Floer complex $CF(L)$ is defined as a free module over  $\mathbb{Z}_2$ generated by
all points of the intersection $L \cap \varphi _1(L)$.
Consider a time-dependent family 
$J=\{J_t\}_{0\leq t\leq 1}$ 
of almost complex structures on $M$ compatible with $\omega$. 
Floer's boundary operator $\partial_J : CF(L)\rightarrow CF(L)$ is defined by
counting (modulo 2) the isolated {\em $J$-holomorphic strips} $v:{\mathbb R} \times [0,1] \to M$ 
connecting pairs of points of $L\cap \varphi _1(L)$ with boundary in $L\cup \varphi _1(L)$.
For a generic choice of $(H, J)$, the homology  
$HF(L):=H_*(CF(L), \partial_J)$ 
is well-defined and called  the {\it Floer homology} of $L$ with $\mathbb{Z}_2$-coefficient,
which is invariant under the Hamiltonian isotopies of $L$ (see \cite{Floer88}, \cite{OhI} and \cite{Oh96} for more details).

Let us regard the intersection $\mathcal{C}:=L \cap \varphi_1(L)$ as points in $L$.
For any two points $p, q \in \mathcal{C}$, consider an isolated $J$-holomorphic strip 
$v:{\mathbb R} \times [0,1] \to M$ from $p$ to $q$, 
which defines a path $\gamma$ in $L$ from $p$ to $q$ by $\gamma(s):=v(s,0)$.
We denote by $\Gamma$ all the paths $\gamma$ given by this procedure.
Notice that to reconstruct the Floer complex $(CF(L),\partial_J)$ 
it is enough to know the above collection $(\mathcal{C},\Gamma)$ of points and paths.

\subsection{Damian's spectral sequence}

Now we recall the definition of Damian's lifted Floer homology for monotone Lagrangian submanifolds 
(see \cite{Damian} for details).
Let $(M,\omega)$ be a closed symplectic manifold and $L\subset M$ be 
a closed embedded monotone Lagrangian submanifold with $N_L\geq 3$.
We start with the data $(\mathcal{C},\Gamma)$ and fix an arbitrary covering $\pi:\bar{L}\rightarrow L$.
For any point $p \in \mathcal{C}$, denote by $\{ p_i \}_{i \in I}$ the elements of the fiber  $\pi^{-1}(p)$.
Consider the set $\bar{\Gamma}$ of all the lifts of the paths of $\Gamma$ to the covering space $\bar{L}$.
For any points $p_i,q_j \,(i,j \in I)$, the cardinality of elements in $\bar{\Gamma}$ which connect $p_i$ 
with $q_j$ is finite.
We denote by $n(p_i,q_j)$ its parity.
Let $CF^{\bar{L}}(L)$ be the free $\mathbb Z_2$-module generated by $\bigcup_{p \in \mathcal{C}}\pi^{-1}(p)$.
One can define the boundary operator $\partial^{\bar{L}}$ on $CF^{\bar{L}}(L)$ by the formula
\[
\partial^{\bar{L}}(p_i)=\sum_{\pi(q_j)=q \in \mathcal{C}} n(p_i,q_j)q_j. 
\]
Damian \cite[Proposition 2.6]{Damian} proved that $(CF^{\bar{L}}(L),\partial^{\bar{L}})$ 
is a complex and its homology $HF^{\bar{L}}(L):=H_*(CF^{\bar{L}}(L),\partial^{\bar{L}})$ 
is called the {\it lifted Floer homology} of $L$, which is also invariant under the Hamiltonian isotopies of $L$.

The following spectral sequence is an adaptation of Biran's construction (see \cite[Theorem 5.2.A]{Biran06}) 
to the lifted Floer homology, which is the main technical tool used in the sequel sections.

\begin{theorem}[Damian's spectral sequence, {\cite[Theorem 2.9]{Damian}}] \label{Th:Damian}
Denote by $\Lambda=\mathbb{Z}_2[T, T^{-1}]$ the algebra of Laurent polynomials over $\mathbb{Z}_2$ 
and $\Lambda^i\subset \Lambda$ the subspace of homogeneous elements of degree $i$. 
There exists a spectral sequence  $\{E^{p,q}_{r}, d_r\}$ which satisfies the following properties:
\begin{enumerate}
\item $E^{p,q}_0=C^{\bar{L}}_{p+q-pN_L} \otimes \Lambda^{pN_L}$, $d_0=[\partial_0^{\bar{L}}]\otimes 1$. 
\item $E^{p,q}_1=H_{p+q-pN_L}(\bar{L},\mathbb{Z}_2)\otimes \Lambda^{pN_L}$, $d_1=[\partial_1^{\bar{L}}]\otimes T^{-N_L}$,
where 
\[
[\partial_1^{\bar{L}}]: H_{p+q-pN_L}(\bar{L};\mathbb{Z}_2)\rightarrow H_{p+q-1-(p-1)N_L}(\bar{L};\mathbb{Z}_2)
\]
is induced by $\partial_1^{\bar{L}}$.
\item 
For any $r\geq 1$, $E^{p,q}_r$ has the form $E^{p,q}_r=V^{p,q}_r\otimes \Lambda^{pN_L}$ with 
$d_r=\delta_r\otimes T^{-rN_L}$, where each
$V^{p,q}_r$ is a vector space over $\mathbb{Z}_2$ and $\delta_r: V^{p,q}_r\rightarrow V^{p-r, q+r-1}_r$ is 
a homomorphism defined for every $p, q$ and satisfies
$\delta_r\circ \delta_r=0$. Moreover, 
\[
V^{p,q}_{r+1}=
\frac{\mathrm{Ker}(\delta_r: V^{p,q}_r \rightarrow V^{p-r, q+r-1}_r)}
{\mathrm{Im}(\delta_r: V^{p+r,q-r+1}_r \rightarrow V^{p,q}_r)}.
\]
In particular,  
$V^{p,q}_0=C^{\bar{L}}_{p+q-pN_L}$, 
$V_1^{p,q}=H_{p+q-pN_L}(\bar{L};\mathbb{Z}_2)$, $\delta_1=[\partial_1^{\bar{L}}]$. 
\item 
$E^{p,q}_r$ collapses at $(\nu+1)$-step and for any $p\in \mathbb{Z}$, 
$\oplus_{q\in \mathbb{Z}} E^{p,q}_{\infty} \cong HF^{\bar{L}}(L)$, 
where $\nu=[\frac{\dim L +1}{N_L}]$. 
\end{enumerate}
\end{theorem}

Back to the Gauss image, 
since $\nu=[\frac{\dim L +1}{N_L}]=[\frac{(n+1)g}{2n}]$,
one can easily get

\begin{corollary} 
For a Gauss image $L^{n}=\mathcal{G}(N^n)\subset Q_n(\mathbb{C})$, $g\geq 3$ 
and any $p, q\in \mathbb{Z}$, we have
\begin{enumerate}
\item $E^{p,q}_2=E^{p,q}_{\infty}$ if and only if $g=3$ and $(m_1,m_2)=(2,2), (4,4), (8,8)$.
\item  $E^{p,q}_3=E^{p,q}_{\infty}$ if and only if $g=3, (m_1,m_2)=(1,1)$ or  $g=4$.
\item $E^{p,q}_4=E^{p,q}_{\infty}$ if and only if $g=6, (m_1,m_2)=(1,1)$ or $(2,2)$.
\end{enumerate}
\end{corollary}

\section{Topology of Gauss images of isoparametric hypersurfaces %
 with $g=3$}
\label{Sec: g=3}

An isoparametric hypersurface with $g=3$,  
the so called Cartan hypersurface, is given by a tube around the standard embedding 
of the projective plane $\mathbb{F}P^{2}$ in $S^{3d+1}$, 
where $\mathbb{F}$ is the division algebra $\mathbb{R}$, $\mathbb{C}$, $\mathbb{H}$ or $\mathbb{O}$, 
and $d=1,2,4$ or $8$ \cite{Cartan}.  
The multiplicity satisfies $m=m_1=m_2=d$, and $n=3m$. 
In this section, we first observe the topology of  Gauss images of isoparametric hypersurfaces with 
$g=3$, then apply Biran and Cornea's general fact 
to obtain their wideness except for the case $n=3$. 

\begin{lemma}\label{3dim Z_2 sphere}
The Gauss image $L^3$ of $g=3$ is a $\mathbb{Z}_2$-homology sphere.
\end{lemma}

\begin{proof}
Recall that all homogeneous isoparametric hypersurfaces $N^n \subset S^{n+1}(1)$ can be obtained as principal orbits of compact Riemannian symmetric pairs $(U, K)$ of 
rank $2$, and that 
isoparametric hypersurfaces with $g=3$ must be homogenous.
For the case when $(g,m,n)=(3,1,3)$,  
in \cite[p.780]{Ma-Ohnita1}
we have a homogenous space expression  
$N=K/K_0 \cong \frac{SO(3)}{\mathbb{Z}_2\oplus \mathbb{Z}_2}\cong SU(2)/\tilde{K}_0$, 
and 
$L=K/K_{[\mathfrak a]}\cong SU(2)/\tilde{K}_{[\mathfrak{a}]}$, 
where 
$\tilde{K}_{[\mathfrak{a}]}$ is generated by $\tilde{K}_0$ and 
\[
\pm 
\left(
\begin{array}{cc}
\frac{1+\sqrt{-1}}{2} & \frac{1+\sqrt{-1}}{2} \\
\frac{-1+\sqrt{-1}}{2} & \frac{1-\sqrt{-1}}{2}\\
\end{array}
\right)
=:\pm A_0,
\quad 
\pm 
\left(
\begin{array}{cc}
\frac{-1+\sqrt{-1}}{2}&\frac{1+\sqrt{-1}}{2} \\
\frac{-1+\sqrt{-1}}{2}&\frac{-1-\sqrt{-1}}{2}\\
\end{array}
\right)
=: \pm A_0^2,
\]
with 
$A_0^3=-\left(
\begin{array}{cc}
1&0 \\
0&1\\
\end{array}
\right)$.
By a direct computation, we obtain $[\tilde{K}_{[\mathfrak{a}]},\tilde{K}_{[\mathfrak{a}]}]=\tilde{K}_{0}$.
Hence
\[
H_1(L^3; \mathbb{Z})\cong \pi_1(L^3)/[\pi_1(L^3), \pi_1(L^3)]\cong \tilde{K}_{[\mathfrak{a}]}/[\tilde{K}_{[\mathfrak{a}]}, \tilde{K}_{[\mathfrak{a}]}]\cong \tilde{K}_{[\mathfrak{a}]}/\tilde{K}_{0}
 \cong \mathbb{Z}_3.
\]
Since $L^{3}$ is orientable, by the Poincar\'{e} duality 
and universal coefficient theorem, 
we obtain
\[
H_k(L^3; \mathbb{Z})\cong 
\left\{
\begin{array}{ll}
\mathbb{Z} & \hbox{ if } k=0, 3, \\
\mathbb{Z}_3 & \hbox{ if } k=1, \\
0 & \hbox{ if } k=2.
\end{array}
\right.
\]
Therefore, by the universal coefficients theorem, 
$L^3$ is a $\mathbb{Z}_2$-homology sphere. 
\end{proof}

\begin{lemma} \label{Z_2 sphere}
The Gauss image $L^n=\mathcal{G}(N^n)$ of a Cartan isoparametric hypersurface 
$N^n\subset S^{n+1}(1)$ \rm{(}i.e., $g=3$ and $n=3, 6, 12, 24$\rm{)} is a $\mathbb{Z}_2$-homology sphere. 
\end{lemma}

\begin{proof}
 The $\mathbb{Z}_2$-homology of $N^n$ is known  \cite{Muenzner1} as
\[
H_k(N; \mathbb{Z}_2)\cong \left\{
                        \begin{array}{ll}
                          \mathbb{Z}_2, & \hbox{ for } k=0, n,\\
                        \mathbb{Z}_2\oplus \mathbb{Z}_2, &  \hbox{ for } k=m, 2m,\\
                        0, & \hbox{ otherwise, }
                        \end{array}
                      \right.
\]
where $n=3m$ and 
$m\in \{1,2,4,8\}$. 
Since $\mathcal{G}: N^n\rightarrow L^n$ is a covering map and $\mathbb{Z}_3$ acts on $N^n$ as a Deck transformation, 
Theorem 2.4 in \cite[p.120]{Bredon} implies 
\[
H_k(L^n ;\mathbb{Z}_2)=H_k(N/\mathbb{Z}_3;\mathbb{Z}_2)\cong H_k(N;\mathbb{Z}_2)^{\mathbb{Z}_3}=0,
\]
for $k\neq 0, m, 2m, 3m(=n)$.

On the other hand, 
when $m\in \{2,4,8\}$, $n$ is even and
$\chi(L)=\chi(N^n)/3=6/3=2$
holds. 
Putting 
$l:=\dim_{\mathbb{Z}_2} H_m(L;\mathbb{Z}_2)$, 
we obtain $2=\chi(L)=1+l+l+1=2+2l$, 
which yields  $l=0$. 
Hence $L$ is a $\mathbb{Z}_2$-homology sphere.

Combing with Lemma \ref{3dim Z_2 sphere}, we complete the proof.
\end{proof}

Let us give several applications of Lemma \ref{Z_2 sphere}. 
Firstly we recall Biran and Cornea's result 
\cite{Biran-Cornea} 
stated in \cite[Proposition 1.6 (3)]{Biran-Khanevsky13}:
\begin{proposition}
\label{Prop:Biran-Cornea} 
Let $L$ be a closed monotone Lagrangian submanifold embedded 
in a closed symplectic manifold $(M,\omega)$.
If $H_i(L;\mathbb{Z}_{2})=0$ for any $i\in \mathbb{Z}$ with $i\equiv -1\, \mathrm{mod}\ N_L$, 
then $HF(L)\cong H_*(L;\mathbb{Z}_2)\otimes \Lambda$.
\end{proposition}
By Proposition \ref{MinMaslovGaussImage}, Lemma \ref{Z_2 sphere} and Proposition \ref{Prop:Biran-Cornea}, 
we obtain 
\begin{corollary}\label{cor:g=3}
Let $L^n=\mathcal{G}(N^n)\subset Q_n(\mathbb{C})$  be the Gauss image of 
an isoparametric hypersurface $N$ with $g=3$ and $n=6,12,24$.
Then $L$ is wide.
Hence, 
$\#(L\cap\varphi(L))\geq 2$ for any $\varphi \in \mathrm{Ham}(Q_n(\mathbb{C}), \omega_{\mathrm{std}})$ 
such that $L$ intersects with $\varphi(L)$ transversally.
\end{corollary}

Secondly, recall Biran and Khanevsky's result:

\begin{proposition}[{\cite[Corollary 1.5]{Biran-Khanevsky13}}]
\label{pro:B-K1}
Let $L^n$ be a closed monotone Lagrangian submanifold 
embedded in $Q_n(\mathbb{C})$ 
with $N_L \geq 2$ and $n \geq 2$.
If $HF(L) \neq 0$, then $L \cap S^n \neq \emptyset$, 
where $S^n$ denotes the totally geodesic Lagrangian sphere of $Q_n(\mathbb{C})$.
\end{proposition}

Combining Corollary \ref{cor:g=3} with Proposition \ref{pro:B-K1}, we obtain

\begin{corollary}\label{cor: LSn}
Let $L^n=\mathcal{G}(N^n)\subset Q_n(\mathbb{C})$  be as in Corollary \ref{cor:g=3}. 
Then 
$\varphi(L) \cap S^n \neq \emptyset$
for any $\varphi \in \mathrm{Ham}(Q_n(\mathbb C), \omega_{\mathrm{std}})$.
\end{corollary}

For any $n$-dimensional submanifold $R$ in $Q_n(\mathbb{C})$, it holds 
\[
\int_{SO(n+2)} \#(\phi (R) \cap S^n)d\mu_{SO(n+2)}(\phi)  \leq  2 \frac{\mathrm{Vol}(SO(n + 2))}{\mathrm{Vol}(S^n)} 
\mathrm{Vol}(R),
\]
where $\phi\in SO(n+2)\subset \mathrm{Ham}(Q_n(\mathbb{C}), \omega_{\mathrm{std}})$ 
((6.7)  in \cite{IST}).
Finally, we give an estimate of the volume of Gauss images 
deformed by Hamiltonian isotopies of $Q_n(\mathbb{C})$.
Namely, applying above to $R=\varphi(L)$, we obtain
\begin{corollary}
For a Lagrangian submanifold  $L$  in Corollary \ref{cor:g=3},
\[
\frac{1}{2}\, \mathrm{Vol}(S^n)\leq\mathrm{Vol}(\varphi(L))
\]
holds for any  $\varphi \in \mathrm{Ham}(Q_n(\mathbb C), \omega_{\mathrm{std}})$.
\end{corollary} 

\begin{remark} 
\begin{enumerate}
\item[(1)]
From Gorodski and Podesta's  classification \cite{Gorodski-Podesta} of all compact tight 
$\mathbb{Z}_2$-homology Lagrangian spheres in simply connected compact homogeneous K\"{a}hler manifolds, 
we know that any compact tight 
$\mathbb{Z}_2$-homology Lagrangian sphere in $Q_n(\mathbb{C})$ is the standard real form $S^n$ up to congruence.
\item[(2)]
Biran and Cornea \cite{Biran-Cornea} proved that any even dimensional 
closed Lagrangian submanifold $L$ embedded   
in $Q_n(\mathbb{C})$ with $H_1(L;\mathbb{Z})=0$ must be 
a $\mathbb{Z}_2$-homology sphere.
Notice that the Gauss images $L$ with $g=3$ 
are $\mathbb{Z}_2$-homology spheres 
with $H_1(L;\mathbb{Z})\cong \mathbb{Z}_3$.
\end{enumerate}
\end{remark}

\section{Hamiltonian Non-displaceability of Gauss images of isoparametric hypersurfaces with $g=4$}
\label{Sec: g=4}

Now we consider the Gauss images 
$L$
of isoparametric hypersurfaces with $g=4$.
Since $N_L=\frac{2n}{4}=\frac{n}{2}$,
$N_L\geq 3$ leads to $n\geq 6$. 
 Thus except for the only one case $n=4$ when 
$L=\frac{SO(2)\times SO(3)}{(\mathbb{Z}_2\times 1)\mathbb{Z}_4}$,
the argument in Subsection 3.2 can work 
and the lifted Floer homology $HF^{\bar{L}}(L)$ is well-defined for any covering $\bar{L}$ of $L$.

\begin{proposition} \label{prop: g=4}
Let $L=\mathcal{G}(N)\subset Q_n(\mathbb{C})$ be the Gauss image of 
an isoparametric hypersurface $N$ with $g=4$ and $m_1\geq 2$. 
Then $HF^{\bar{L}}(L)\neq 0$ for $\bar{L}=N$.
\end{proposition}

\begin{proof}
Consider the $\mathbb{Z}_4$-covering $\bar{L}=N^n$ over $L^n$.
We argue by contradiction.
Assume that 
$HF^{\bar{L}}(L)=0$. 
Since $\nu=[\frac{\dim L +1}{N_L}]=2$,  
Damian's spectral sequence $\{E^{p,q}_r, d_r\}$ collapses at $3$-step.
Then
\[
0=E^{0,q}_3=V^{0,q}_3
=\frac{\mathrm{Ker}(\delta_2: V^{0,q}_2 \rightarrow V^{-2,q+1}_2)}{\mathrm{Im}(\delta_2: V^{2,q-1}_2\rightarrow V^{0,q}_2)}
\]
and thus we have the following exact sequence
\begin{equation}\label{seq:g=4}
V^{2,q-1}_2\rightarrow V^{0,q}_2\rightarrow V^{-2,q+1}_2
\end{equation}
for any $q\in \mathbb{Z}$.
Since \begin{eqnarray*}
V^{2,q-1}_2&=&\frac{\mathrm{Ker}([\partial_1^{\bar{L}}]: H_{q+1-2N_L}(\bar{L};\mathbb{Z}_2) \rightarrow H_{q-N_L} (\bar{L};\mathbb{Z}_2) )}
{\mathrm{Im}([\partial_1^{\bar{L}}]: H_{q+2-3N_L}(\bar{L};\mathbb{Z}_2)\rightarrow H_{q+1-2N_L}(\bar{L};\mathbb{Z}_2))},\\
V^{-2,q+1}_2&=&\frac{\mathrm{Ker}([\partial_1^{\bar{L}}]: H_{q-1+2N_L}(\bar{L};\mathbb{Z}_2) \rightarrow H_{q-2+3N_L}(\bar{L};\mathbb{Z}_2))}
{\mathrm{Im}([\partial_1^{\bar{L}}]: H_{q+N_L}(\bar{L};\mathbb{Z}_2)\rightarrow H_{q-1+2N_L}(\bar{L};\mathbb{Z}_2))},
\end{eqnarray*} 
we see that $V^{2,q-1}_2=V^{-2,q+1}_2=0$ when $2\leq q\leq n-2$.
Then it follows from \eqref{seq:g=4} that 
\begin{eqnarray*}
0=V^{0,q}_2
&=&\frac{\mathrm{Ker}([\partial_1^{\bar{L}}]: H_{q}(\bar{L};\mathbb{Z}_2) \rightarrow H_{q-1+N_L}(\bar{L};\mathbb{Z}_2))}
{\mathrm{Im}([\partial_1^{\bar{L}}]: H_{q+1-N_L}(\bar{L};\mathbb{Z}_2)\rightarrow H_{q}(\bar{L};\mathbb{Z}_2))}
\end{eqnarray*} 
holds for $2\leq q\leq n-2$.

Putting $q=N_L=m_1+m_2$, we have the exact sequence
\begin{equation}\label{eq:seq g=4}
H_{1}(\bar{L};\mathbb{Z}_2)\rightarrow H_{m_1+m_2}(\bar{L}; \mathbb{Z}_2)\rightarrow H_{2(m_1+m_2)-1}(\bar{L}; \mathbb{Z}_2).
\end{equation}

Recall that by M\"{u}nzner's result \cite{Muenzner1} the $\mathbb{Z}_2$-homology of $N^n$ is given by
\[
H_k(N; \mathbb{Z}_2)\cong \left\{
                        \begin{array}{ll}
                          \mathbb{Z}_2, & \hbox{ for } k=0, m_1, m_2, 2m_1+m_2, m_1+2m_2, n,\\
                        \mathbb{Z}_2\oplus \mathbb{Z}_2, &  \hbox{ for } k=m_1+m_2,\\
                        0, & \hbox{ otherwise.}
                        \end{array}
                      \right.
\]

By the assumption that $m_1\geq 2$, 
(\ref{eq:seq g=4})  leads to the  exact sequence
$0\rightarrow \mathbb{Z}_2\oplus \mathbb{Z}_2 \rightarrow 0,$
which is a contradiction. 
\end{proof}

\begin{corollary} 
Under the same assumption as in Proposition \ref{prop: g=4}, 
the Gauss image $L^n\subset (Q_n(\mathbb{C}), \omega_{\mathrm{std}})$ is Hamiltonian non-displaceable.
\end{corollary}

\begin{remark}
Note that 
any non-homogeneous isoparametric hypersurface in the standard sphere 
satisfies the assumption of Proposition \ref{prop: g=4}. 
By R.~Takagi's result \cite{Takagi} 
we know that an isoparametric hypersurface $N^{n}$ with $g=4$ and $m_{1}=1$ 
must be $N^{n}=\frac{SO(2)\times SO(k+2)}{{\mathbb Z}_{2}\times SO(k)}\ (n=2k+2, m_{2}=k)$, 
in particular homogeneous. 
\end{remark}

\section{Hamiltonian Non-displaceability of 
the Gauss image of an isoparametric hypersurface with $g=6$}
\label{Sec: g=6}

When $g=6$, $N_L\ge3$ implies $m_1=m_2=2$ \cite{Abresch}.  Thus we focus on the Gauss image $L^{12}=\mathcal{G}(N^{12})$ of 
the isoparametric hypersurface $N^{12}\subset S^{13}(1)$ with $g=6$ and $m_1=m_2=2$.  
Compared with the case when $g=4$, this case is more complicated because $\{E^{p,q}_r, d_r \}$ 
collapses at $4$-step.

\begin{proposition} \label{prop: g=6}
Let $L^{12}=\mathcal{G}(N^{12})\subset Q_{12}(\mathbb{C})$ be the Gauss image of 
an isoparametric hypersurface with $g=6$ and $m_1=m_2=2$. Then $HF^{\bar{L}}(L)\neq 0$ for $\bar{L}=N$.
In particular, $L$ is Hamiltonian non-displaceable in $Q_{12}(\mathbb{C})$.
\end{proposition}

\begin{proof}
Take the $\mathbb{Z}_6$-covering $\bar{L}=N^{12}$ over $L^{12}$.
Since $\nu=[\frac{12+1}{4}]=3$, Damian's spectral sequence $\{E^{p,q}_r, d_r\}$ collapses at $4$-step.
Assume that $HF^{\bar{L}}(L)=0$. Then we have
\[
0=E^{0,q}_4=V^{0,q}_4
=\frac{\mathrm{Ker}(\delta_3: V^{0,q}_3 \rightarrow V^{-3,q+2}_3)}{\mathrm{Im}(\delta_3: V^{3,q-2}_3\rightarrow V^{0,q}_3)}
\]
and the following exact sequence
\begin{equation}\label{eq:V_3}
V^{3,q-2}_3\rightarrow V^{0,q}_3\rightarrow V^{-3,q+2}_3
\end{equation}
for any $q\in \mathbb{Z}$. 
 
Notice that
\begin{eqnarray*}
V^{3,q-2}_3&=&\frac{\mathrm{Ker}(\delta_2: V^{3,q-2}_2\rightarrow V^{1,q-1}_2)}
{\mathrm{Im}(\delta_2: V^{5,q-3}_2\rightarrow V^{3,q-2}_2)},\\
V^{3,q-2}_2
&=&\frac{\mathrm{Ker}([\partial_1^{\bar{L}}]: H_{q-11}(\bar{L};\mathbb{Z}_2) \rightarrow H_{q-8}(\bar{L};\mathbb{Z}_2) )}
{\mathrm{Im}([\partial_1^{\bar{L}}]: H_{q-14}(\bar{L};\mathbb{Z}_2) \rightarrow H_{q-11}(\bar{L};\mathbb{Z}_2) )}.
\end{eqnarray*} 
Consequently, $V^{3,q-2}_3=0$ when  $q\leq 10$. 
Similarly, from
\begin{eqnarray*}
V^{-3,q+2}_3&=&\frac{\mathrm{Ker}(\delta_2: V^{-3,q+2}_2\rightarrow V^{-5,q+3}_2)}
{\mathrm{Im}(\delta_2: V^{-1,q+1}_2\rightarrow V^{-3,q+2}_2)},\\
V^{-3,q+2}_2
&=&\frac{\mathrm{Ker}([\partial_1^{\bar{L}}]: H_{q+11}(\bar{L};\mathbb{Z}_2) \rightarrow H_{q+14}(\bar{L};\mathbb{Z}_2))}
{\mathrm{Im}([\partial_1^{\bar{L}}]: H_{q+8}(\bar{L};\mathbb{Z}_2)\rightarrow H_{q+11}(\bar{L};\mathbb{Z}_2))},
\end{eqnarray*} 
we see that $V^{-3,q+2}_3=0$ if  $q\geq 2$.
Thus (\ref{eq:V_3}) leads to
$V^{0,q}_3=0$ for $ 2\leq q\leq 10$.
Due to
\[
0=V^{0,q}_3
=\frac{\mathrm{Ker}(\delta_2: V^{0,q}_2 \rightarrow V^{-2,q+1}_2)}{\mathrm{Im}(\delta_2: V^{2,q-1}_2\rightarrow V^{0,q}_2)},
\]
 we get the exact sequence
$ V^{2,q-1}_2\rightarrow V^{0,q}_2 \rightarrow V^{-2,q+1}_2$
 for any $2\leq q\leq 10$.
Thanks to
\begin{eqnarray*}
V^{2,q-1}_2
&=&\frac{
\mathrm{Ker}
([\partial_1^{\bar{L}}]: H_{q-7} (\bar{L};\mathbb{Z}_2)\rightarrow H_{q-4} (\bar{L};\mathbb{Z}_2))}
{\mathrm{Im}([\partial_1^{\bar{L}}]: H_{q-10} (\bar{L};\mathbb{Z}_2)\rightarrow H_{q-7} (\bar{L};\mathbb{Z}_2))},\\
V^{-2,q+1}_2
&=&\frac{\mathrm{Ker}([\partial_1^{\bar{L}}]: H_{q+7} (\bar{L};\mathbb{Z}_2) \rightarrow H_{q+10} (\bar{L};\mathbb{Z}_2))}
{\mathrm{Im}([\partial_1^{\bar{L}}]: H_{q+4} (\bar{L};\mathbb{Z}_2)\rightarrow H_{q+7} (\bar{L};\mathbb{Z}_2))},
\end{eqnarray*} 
we derive
$V^{0,6}_2=0$. 
Now because 
\begin{eqnarray*}
0&=&V^{0,6}_2
=\frac{\mathrm{Ker}([\partial_1^{\bar{L}}]: H_{6} (\bar{L};\mathbb{Z}_2) \rightarrow H_{9} (\bar{L};\mathbb{Z}_2))}
{\mathrm{Im}([\partial_1^{\bar{L}}]: H_{3} (\bar{L};\mathbb{Z}_2)\rightarrow H_{6} (\bar{L};\mathbb{Z}_2))},
\end{eqnarray*}
we get the exact sequence 
$H_3(\bar{L};\mathbb{Z}_2) \rightarrow H_6(\bar{L};\mathbb{Z}_2) \rightarrow H_9(\bar{L};\mathbb{Z}_2)$, 
which contradicts 
the facts in \cite{Muenzner1} that $H_3(\bar{L};\mathbb{Z}_2)\cong H_9(\bar{L};\mathbb{Z}_2)=0$ 
and $H_6(\bar{L};\mathbb{Z}_2)\cong \mathbb{Z}_2\oplus \mathbb{Z}_2$. 
\end{proof}

Therefore we complete the proof of Theorem \ref{Main Thm}.

\section{Open problems and a conjecture}
\label{Sec: Open problems}

The following are open problems for further research:

\begin{enumerate}
\item We cannot apply the same method above to determine the Hamiltonian 
non-displaceability of Gauss images of the following isoparametric hypersurfaces:
\begin{itemize}
\item[(i)] $g=3$, $(m_1, m_2)=(1,1)$, $n=3$;
\item[(ii)] $g=4$, $(m_1, m_2)=(1,k)$, $n=2(1+k)$;
\item[(iii)] $g=6$, $(m_1, m_2)=(1,1)$, $n=6$.
\end{itemize}
Consider how we can do in these cases. 

\item 
Determine the Floer homology $HF(L)$ of Gauss images for $(g,n)=(3,3)$ and $g=4, 6$. 

\item 
Describe $H_*(L;\mathbb{Z}_2)$ of the Gauss images $L=N/\mathbb{Z}_g$ for $g=4, 6$ explicitly.

\end{enumerate}

\smallskip\noindent
According to the present work and a private communication with Hajime Ono, 
he and the authors give 
\begin{conjecture}
Any compact connected minimal Lagrangian submanifold in an irreducible Hermitian symmetric space of compact type 
is Hamiltonian non-displaceable.
\end{conjecture}


\end{document}